\documentclass[a4paper,12pt]{amsart}

\usepackage{amssymb,amsmath,comment,eepic}
\usepackage{array}
\usepackage{longtable}
\usepackage{amsfonts}
\usepackage{latexsym}
\usepackage{amssymb}
\usepackage{amsthm}
\usepackage{dsfont}
\usepackage{tikz, subfigure}
\usepackage{verbatim}
\usepackage{graphicx}

\newtheorem{theorem}{Theorem}[section]

\newtheorem{lemma}[theorem]{Lemma}

\newtheorem{remark}[theorem]{Remark}

\newtheorem{definition}[theorem]{Definition}
\newtheorem{conj}[theorem]{Conjecture}

\newcommand{\R}{\mathcal{R}}
\newcommand{\BbN}{\mathbb{N}}
\newcommand{\la}{\lambda}
\newcommand{\Gtrue}{98.3\%}
\newcommand{\Htrue}{91.8\%}
\newcommand{\HGtrue}{91.3\%}
\numberwithin{figure}{section}

\begin{document}
\title[Non-fractal tops]{On a family of self-affine IFS whose attractors have a non-fractal top \\ \today}
\date{\today}
\author[K. G. Hare]{Kevin G. Hare}
\address{Department of Pure Mathematics, University of Waterloo, Waterloo, Ontario, Canada N2L 3G1}
\email{kghare@uwaterloo.ca}
\thanks{Research of K.G. Hare was supported by NSERC Grant 2019-03930}
\author[N. Sidorov]{Nikita Sidorov}
\address{Department of Mathematics, The University of Manchester, Manchester, M13 9PL, United Kingdom}
\email{sidorov@manchester.ac.uk}
\thanks{Research of N. Sidorov was supported by in part by the University of Waterloo}

\maketitle

\begin{abstract}
Let $0< \lambda < \mu<1$ and $\lambda+\mu>1$. In this note we prove that for the vast majority of such parameters the top of the attractor $A_{\lambda,\mu}$ of the IFS $\{(\lambda x,\mu y), (\mu x+1-\mu, \lambda y+1-\lambda)\}$ is the graph of a continuous, strictly increasing function.
Despite this, for most parameters, $A_{\lambda, \mu}$ has a box dimension strictly greater than 1, showing that the upper boundary is not representative
    of the complexity of the fractal.
Finally, we prove that if $\lambda \mu\ge 2^{-1/6}$, then $A_{\lambda,\mu}$ has a non-empty interior.
\end{abstract}

\section{Introduction}

Self-affine iterated function systems (IFS) are well studied. When such an IFS is given by a single matrix, e. g., $\{Mx, Mx+u\}$, it appears that all of its boundary is fractal, though there are no rigorous results in this direction, to our best knowledge. The purpose of this note is to present a family of two-dimensional IFS for which their attractors have a different kind of boundary for the top and the bottom. In particular, their tops are not fractal.

Assume $0< \lambda < \mu<1$ and $\lambda+\mu>1$. Put
\[
T_0(x,y)=(\lambda x, \mu y),\ T_1(x,y)=(\mu x+1-\mu,\lambda y+1-\lambda).
\]
Let $A_{\lambda,\mu}$ denote the attractor for the IFS $\{T_0, T_1\}$.
Notice that $A_{\lambda, \mu} \subset [0,1]\times[0,1]$ -- see Figure~\ref{fig:Alamu}.
Based upon visual inspection of such sets, one would expect that $A_{\lambda, \mu}$ would have
    dimension strictly greater than 1.
Despite this, it also surprisingly appears that the top of this IFS is one-dimensional. This is in stark contrast with the family of IFS $\{(\la x,\mu y),(\la x+1-\la, \mu y+1-\mu)\}$ studied in detail in \cite{HS-ETDS}.

Put
\[
\partial_{top}(A_{\lambda,\mu})=\{(x,y)\in A_{\lambda,\mu} : \forall (x,y')\in
A_{\lambda,\mu}\ \mbox{we have } y'\le y\}.
\]

We will define a closed subset
    $G \subset \{(\lambda, \mu): \lambda + \mu > 1,
    0 < \lambda < \mu < 1 \}$ in Section~\ref{sec:comp} for which $\partial_{top}(A_{\lambda,\mu})$ is strictly increasing and continuous.
This set $G$ has the property that it is at least \Gtrue\ of the 
    parameter space $\{(\lambda, \mu): \lambda + \mu > 1,
    0 <\lambda < \mu < 1 \}$.

We have three main results. The first is 
\begin{theorem}\label{thm1}
For all $(\lambda, \mu) \in G$ we have
    the set $\partial_{top}(A_{\lambda,\mu})$ is the graph of a continuous,
    strictly increasing function.
\end{theorem}

It appears computationally that we can construct a $G$ arbitrarily close to 
    the full parameter space.
From this we make the 
\begin{conj}
For all $0 < \lambda < \mu < 1$ with $\lambda + \mu > 1$
    the set $\partial_{top}(A_{\lambda,\mu})$ is the graph of a continuous,
    strictly increasing function.
\end{conj}

We have
\begin{theorem}\label{thm2}
There exists $(\lambda, \mu)$ with $0 < \lambda < \mu < 1, \lambda+\mu > 1$ such that 
    the set $A_{\lambda,\mu}$ has dimension strictly greater than $1$.
\end{theorem}

In fact Theorem \ref{thm2} is stronger than this.  
We give a range of parameters, making up \Htrue\ of the 
    parameter space $\{(\lambda, \mu): 0< \lambda < \mu < 1, \lambda+\mu>1\}$
    for which $A_{\lambda,\mu}$ has dimension strictly greater than 1.
In fact the range of paramters that satisfy both Theorem \ref{thm1} and 
    \ref{thm2} makes up \HGtrue\ of the parameter space.
Unfortunately the technique used in Theorem \ref{thm2} probably cannot 
    be extended arbitrarily close to 100\%, as we will discuss later.
We observe that if $\lambda \mu < 1/2$ then we necessarily have $\dim(A_{\lambda, \mu}) < 2$ and hence
    all points are boundary points.
This reinforces the observation that the upper boundary of $A_{\lambda,  \mu}$ is not representative of the 
    boundary of $A_{\lambda, \mu}$.

Although the technique does not appear to extend to all parameters $(\lambda, \mu)$, we still believe
\begin{conj}
For all $0 < \lambda < \mu < 1$ with $\lambda + \mu > 1$
    the set $A_{\lambda,\mu}$ has dimension strictly greater than $1$.
\end{conj}

Lastly, using a technique from \cite{HS-2D} we have
\begin{theorem} \label{thm3}
For all $\lambda \mu \geq 2^{-1/6}$ we have $A_{\lambda, \mu}$ has non-empty interior.
\end{theorem}

\begin{figure}
\centering \scalebox{0.4} {\includegraphics{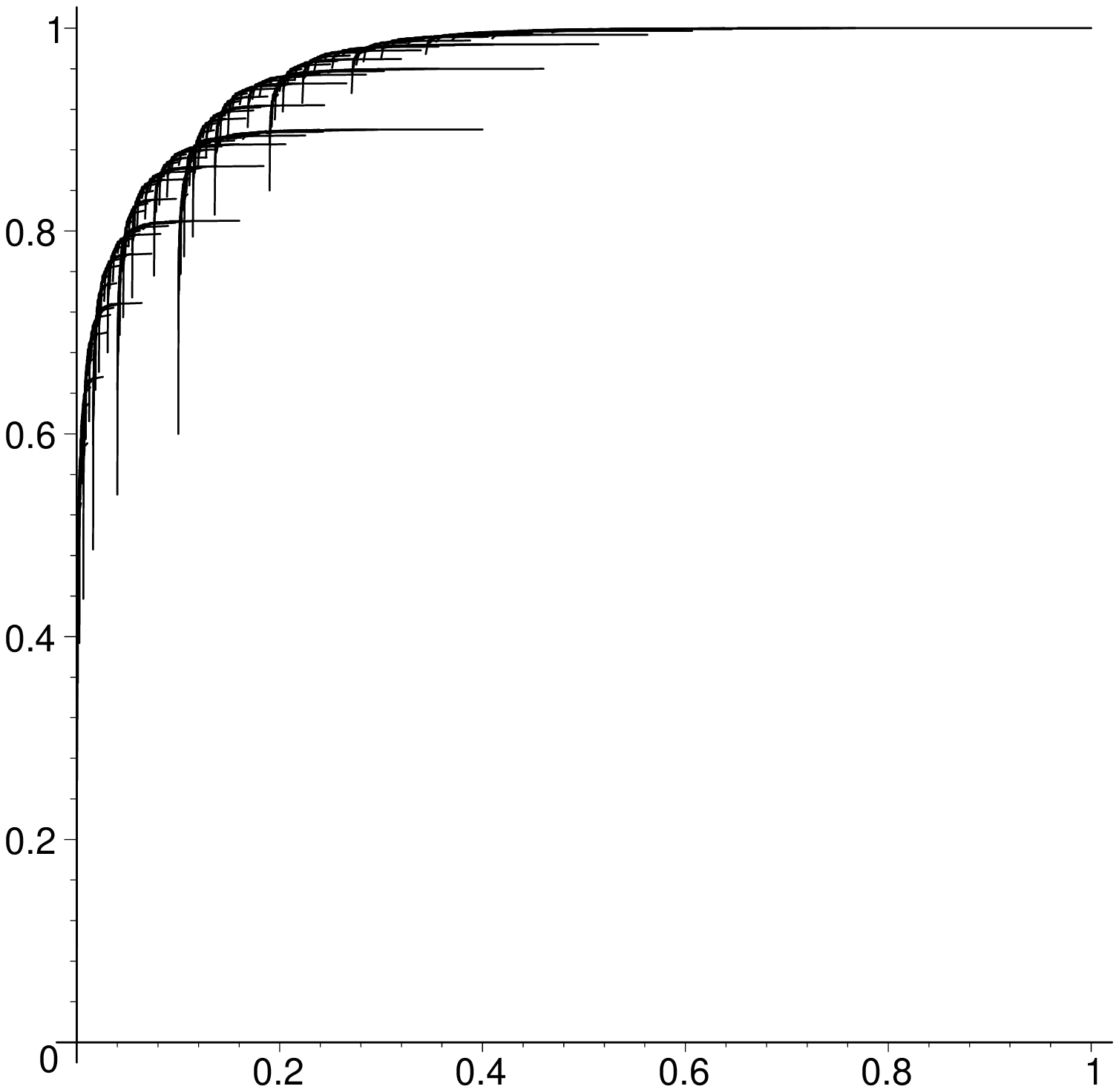}}
\centering \scalebox{0.4} {\includegraphics{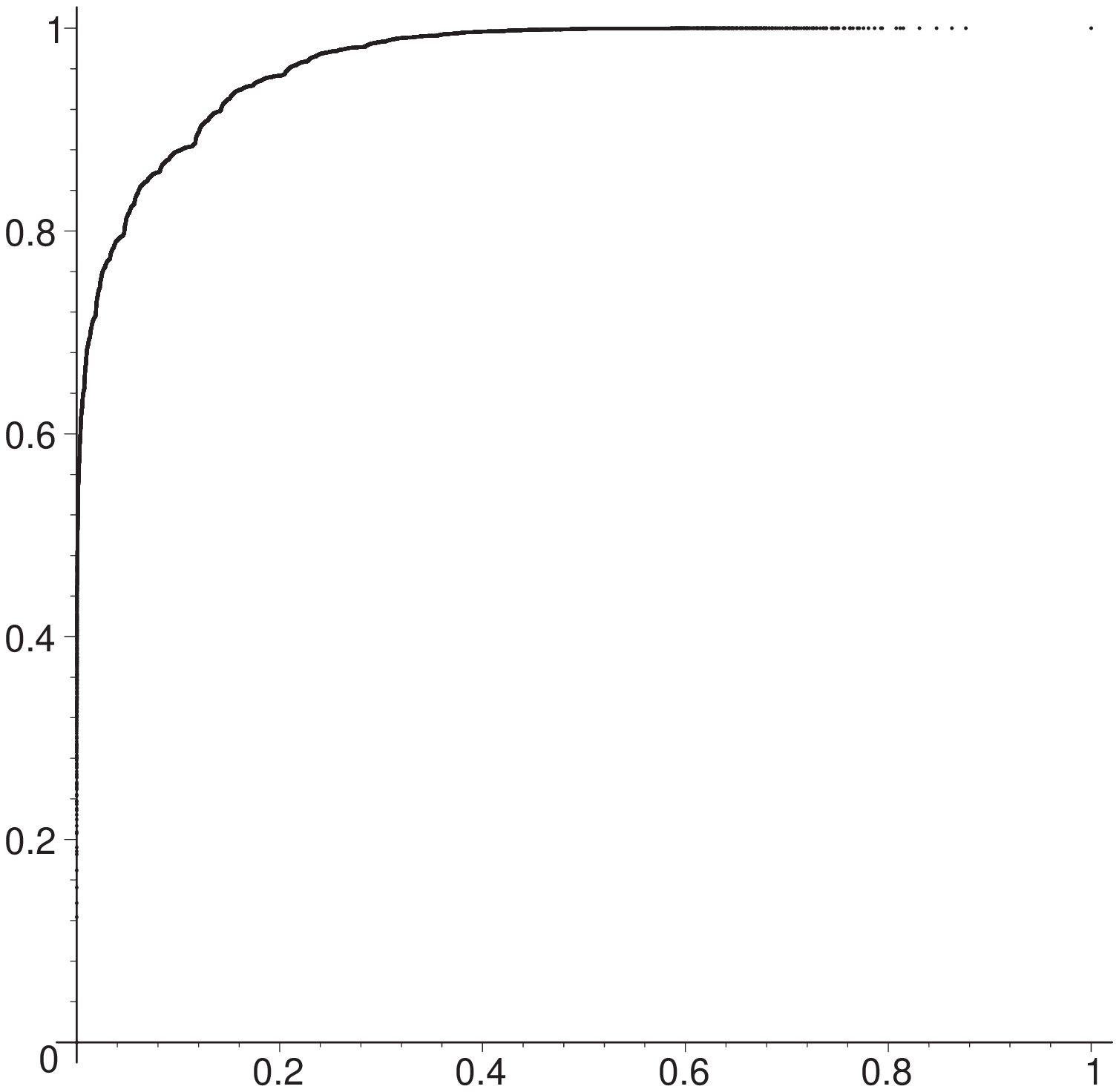}}
\caption{The attractor $A_{0.4,0.9}$ and $\partial_{top}(A_{0.4, 0.9})$.}\label{fig:Alamu}
\end{figure}

In Section \ref{sec:thm1} we give a prove of Theorem \ref{thm1}.
We also introduce a subset $B_{\lambda, \mu} \subset A_{\lambda, \mu}$ upon which the definition of $G$ is based.
A computational investigation of $G$ is given in Section~\ref{sec:comp}.
Sections~\ref{sec:thm2} and \ref{sec:thm3} prove Theorems~\ref{thm2} and \ref{thm3} respectively.

\section{Proof of Theorem \ref{thm1}}
\label{sec:thm1}

We will prove this result in two steps.
The first is to show that $A_{\lambda, \mu}$ contains a strictly increasing
    continuous function going from $(0,0)$ to $(1,1)$ with some additional
    properties.
This will be the set $B_{\lambda, \mu}$ and is described in Lemma \ref{lem:B}.

After this we will introduce a map $\R$ which has
    $\partial_{top}(A_{\lambda,\mu})$ as an attractor, and
    further whose iterates on $B_{\lambda, \mu}$ 
    are continuous increasing functions with the same additional properties as $B_{\lambda, \mu}$.
This is done in Lemma \ref{lem:C}.

This second step requires an additional property on $B_{\lambda, \mu}$
    which conjecturally is true for all $0 < \lambda < \mu < 1$, $\lambda+\mu > 1$,
    and computationally is true for at least \Gtrue\ of such pairs $(\lambda, \mu)$.
The set where this additional property is true is called $G$.
See Definition \ref{defn:G} for a precise definition.

Put
\begin{align*}
  S_0(x,y) & =\begin{cases}
                T_0(x,y), & \mbox{if } \lambda x+\mu y\le1 \\
                (0,0), & \mbox{otherwise}.
              \end{cases}
   \\
   S_1(x,y) & =\begin{cases}
                T_1(x,y), & \mbox{if } \mu x+\lambda y\ge\lambda+\mu-1 \\
                (1,1), & \mbox{otherwise}.
              \end{cases}
\end{align*}
The attractor of $\{S_0, S_1\}$ is not unique.  For example,
    the pair $\{(0,0), (1,1)\}$ is fixed under this map.
It is clear that if we have two different attractors of $\{S_0, S_1\}$, then
    their union is also an attractor.
Further, all attractors are contained in $[0,1] \times [0,1]$.
As such there is a maximal attractor, which we define as
    $B_{\lambda,\mu}$.
Clearly, $B_{\lambda,\mu}\subset A_{\lambda,\mu}$.
\begin{lemma}\label{lem:B}
The attractor $B_{\lambda,\mu}$ is the graph of a continuous function, i.e., for any $x\in[0,1]$ there exists a unique $y\in[0,1]$ such that $(x,y)\in B_{\lambda,\mu}$. This function is strictly increasing.
\end{lemma}

\begin{proof}
Let $X = [0,1] \times [0,1]$.
Notice that $S_0(X)\cap S_1(X)$ is a segment on $x+y=1$. Put
\[
Y_n = \bigcup_{i_1\dots i_n\in\{0,1\}^n}S_{i_1}\dots S_{i_n}(X).
\]
Then $Y_n$ is a union of $2^n$ polygons such that their interiors are disjoint -- see Figure~\ref{fig:S}.
It is worth noting that these polygons may be the points $(0,0)$ or $(1,1)$.
One can show by induction that each non-trivial polygon is either a pentagon or a hexagon. These polygons are ordered: for any two of them, one's upper right corner  is higher than the other's (see Figure~\ref{fig:S}). Also, $S_{i_1}\dots S_{i_n}(X)$ is higher than $S_{j_1}\dots S_{j_n}(X)$
iff $i_1\dots i_n\succ j_1\dots j_n$.

\begin{figure}
\centering \scalebox{0.4} {\includegraphics{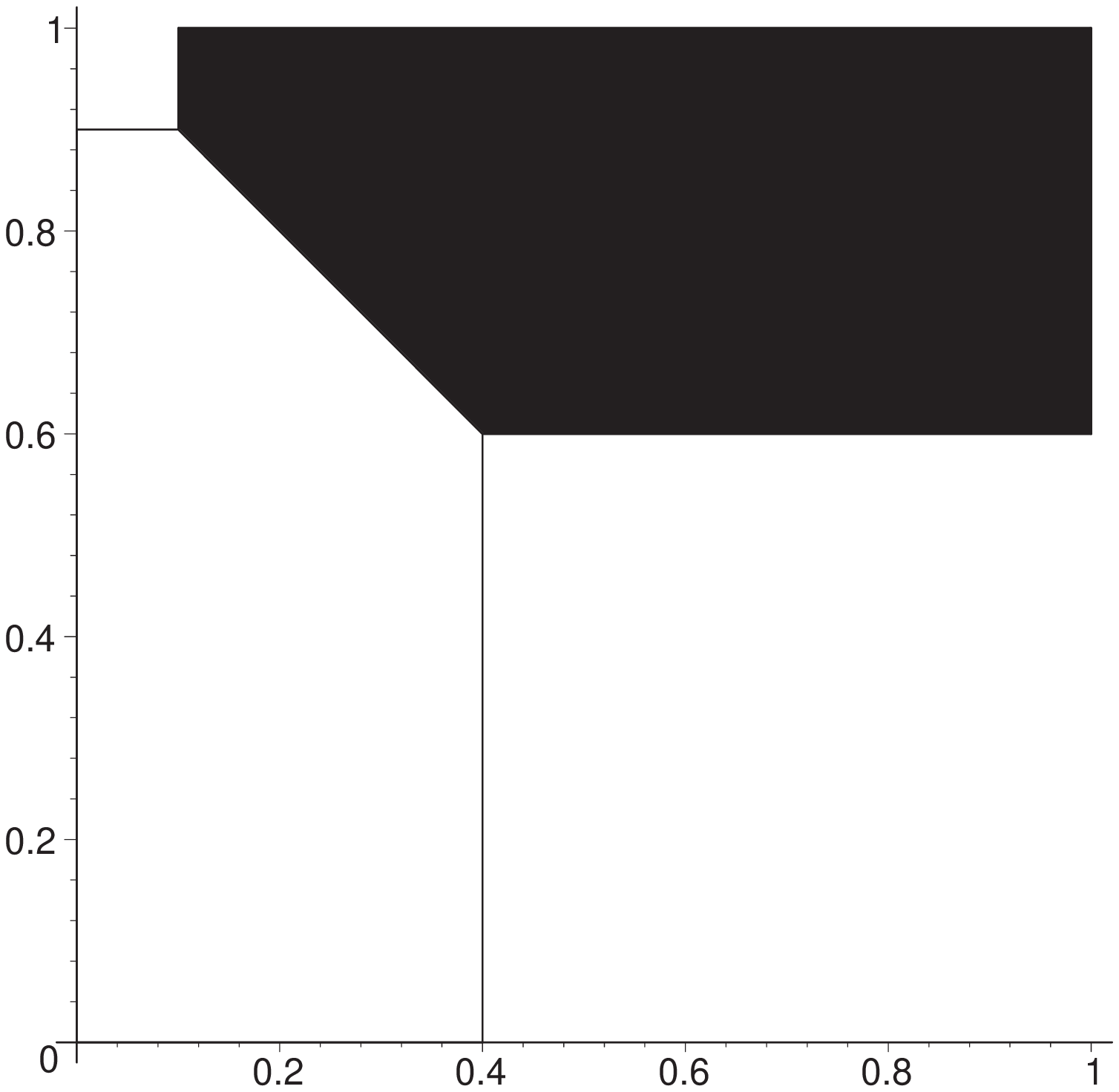}}
\centering \scalebox{0.4} {\includegraphics{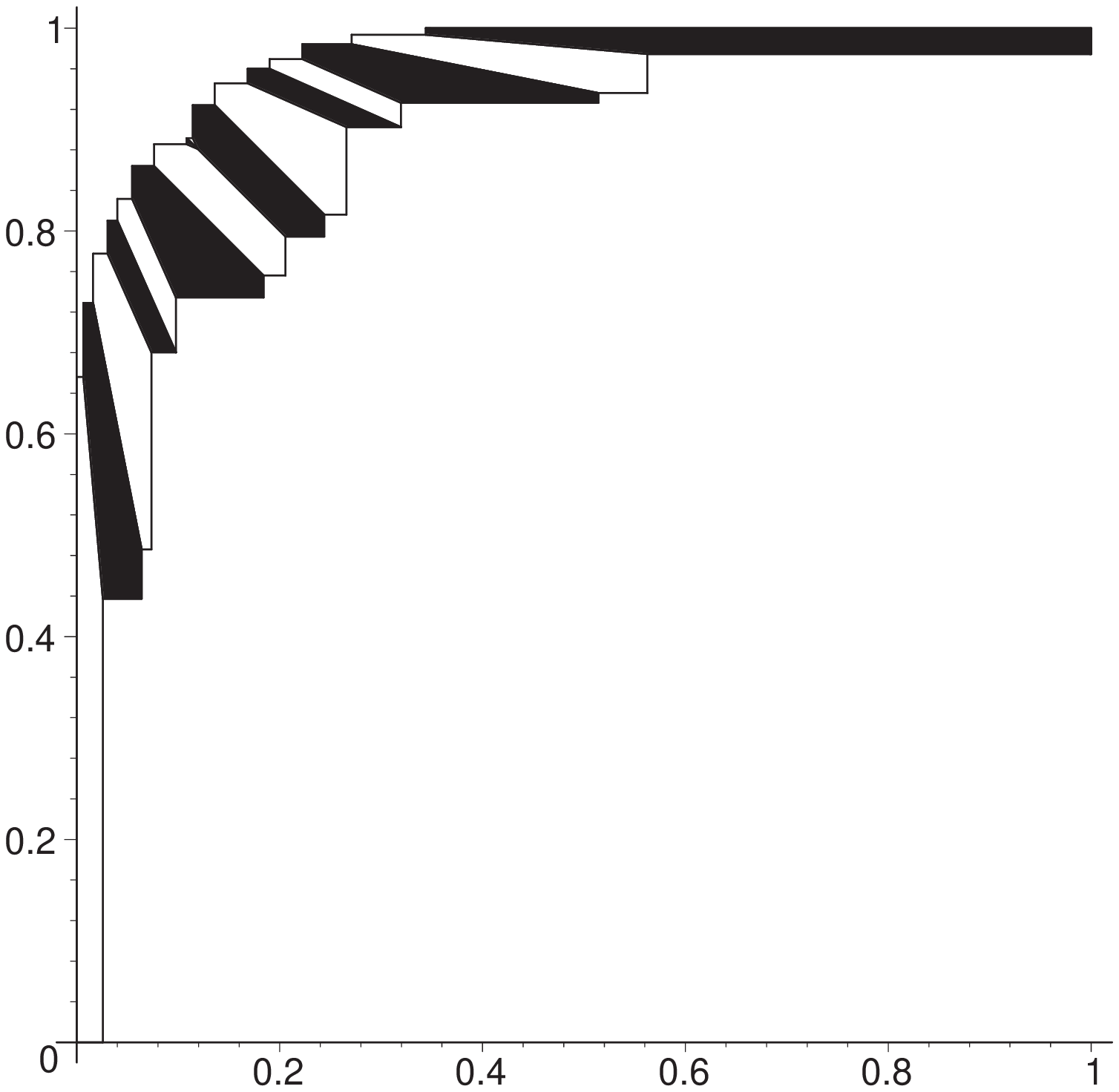}}
\caption{The sets $Y_1$ and $Y_4$ for $\lambda = 0.4, \mu = 0.9$.}\label{fig:S}
\end{figure}

Finally, any intersection of $Y_n$ with any horizontal or vertical line is an interval; this intersection involves only a bounded number of polygons. This follows from the fact that when we go from $Y_n$ to $Y_{n+1}$, we cut out a certain proportion of each polygon both horizontally and vertically -- so we will have that any polygon will be strictly higher and to the right or strictly lower and to the left from any other polygon except a number of them which depends on $\lambda$ and $\mu$ only.

It is easy to see that $Y_n \to B_{\lambda, \mu}$ in the Hausdorff metric.
\end{proof}

A key property of $B_{\lambda, \mu}$ that we will exploit is
        that $T_0(1,1)$ and $T_1(0,0)$ are below $B_{\lambda, \mu}$.
Unfortunately, although this appears to be computationally true for all $0 < \lambda < \mu <1$ with $\lambda + \mu > 1$,
    a general proof is not known.
\begin{definition}
\label{defn:G}
Define
\[ G := \{(\lambda, \mu) : T_0(1,1) \text{ and } T_1(0,0) \text{ are below } B_{\lambda, \mu}\} \]
\end{definition}

In Section \ref{sec:comp} we discuss how one can find regions in $G$, and
    provide a link to data demonstrating that $G$ is at least \Gtrue\ of the $0 < \lambda < \mu <1$ with $\lambda + \mu > 1$.

We now introduce one last function, going from the set of non-empty
    compact sets to non-empty compact sets by
    \[
     \R(A) = \partial_{top}(T_0(A) \cup T_1(A)).
     \]
We observe that $\partial_{top}(A_{\lambda, \mu})$ is fixed by this map.
It is not true in general if $A$ is a continuous function that $\R(A)$ will also
    be a continuous function.

\begin{lemma}
\label{lem:C}
Assume $(\lambda, \mu) \in G$.
Define $R_n = \R^{[n]}(B_{\lambda,\mu})$.
We have
\begin{enumerate}
\item $R_n$ is a continuous increasing function
\label{case:1}
\item $R_{n-1} \leq R_n$ in the sense that
      for $(x,y) \in R_{n-1}$ there exists a $y' \geq y$
      such that $(x,y') \in R_{n}$.
\label{case:2}
\item $R_n \leq \partial_{top}(A_{\lambda, \mu})$.
\label{case:3}
\item $R_n \to \partial_{top}(A_{\lambda,\mu})$ as $n\to\infty$.
\label{case:4}
\item $\partial_{top}(A_{\lambda,\mu})$ has no jump discontinuities and 
    is strictly increasing.
\label{case:5}
\end{enumerate}
\end{lemma}

\begin{proof}
To see \eqref{case:2} and \eqref{case:3},
    observe that
    $R_n \subset \bigcup_{a \in \{0,1\}^n} T_a(B_{\lambda, \mu})
    \subset A_{\lambda,, \mu}$ and
    $\partial_{top}(R_{n-1}) = R_{n-1}$.

We prove \eqref{case:1} by induction.
We observe that $R_0 = B_{\lambda, \mu}$ is a continuous increasing curve
    with the property that $T_{0}(1,1)$ and $T_{1}(0,0)$ are below the
    curve $R_0$.
We see that $R_n \subset T_0 (R_{n-1}) \cup T_1(R_{n-1})$.
Hence $T_0(R_{n-1})$ is a continuous increasing curve from $(0,0)$ to
    $T_0(1,1) = (\lambda, \mu)$.
Further, as $T_0(1,1)$ is below $R_{n-1}$ which in turn is below $R_n$ we
    have that $T_0(1,1)$ is below $T_1(R_{n-1})$.
This implies that the curve $R_n$ is continuous and increasing at
    $x = \lambda$, as $T_1(R_{n-1})$ is continuous and
    increasing at $x = \lambda$.

As similar observation can be made for $T_1(0,0)$.
Hence $R_{n}$ is increasing and continuous.

We have that \eqref{case:4} follows from the observation that
    $R_n = \partial_{top}\left(\bigcup_{a \in \{0,1\}^n} T_a(B_{\lambda, \mu})\right)$ and
    $\lim_{n\to\infty}\left(\bigcup_{a \in \{0,1\}^n} T_a(B_{\lambda, \mu})\right) = A_{\lambda,\mu}$ in the Hausdorff topology.

Lastly, to see \eqref{case:5},
    let $M$ be the supremum of the jump discontinuities of 
    $\partial_{top}(A_{\lambda,\mu})$.
We note that $\lambda M$ is the supremum of the jump discontinuities of
    $\R(\partial_{top}(A_{\lambda,\mu})) = 
        \partial_{top}(A_{\lambda,\mu})$.
Hence $M = 0$ and $\partial_{top}(A_{\lambda,\mu})$ has no jump discontinuities.
As $\partial_{top}(A_{\lambda,\mu})$ is symmetric about $\lambda+\mu=1$ 
    we see that it is strictly increasing.
\end{proof}

\begin{remark}
It is tempting to believe that $\R(B_{\lambda,\mu}) = B_{\lambda,\mu}$.
This is unfortunately not always the case.
In Figure \ref{fig:neq} we show the image of $T_0(B_{0.4, 0.9})$ and $T_1(B_{0.4, 0.9})$, magnified near the region of
    intersection.
\end{remark}

\begin{figure}
\centering \scalebox{0.4} {\includegraphics{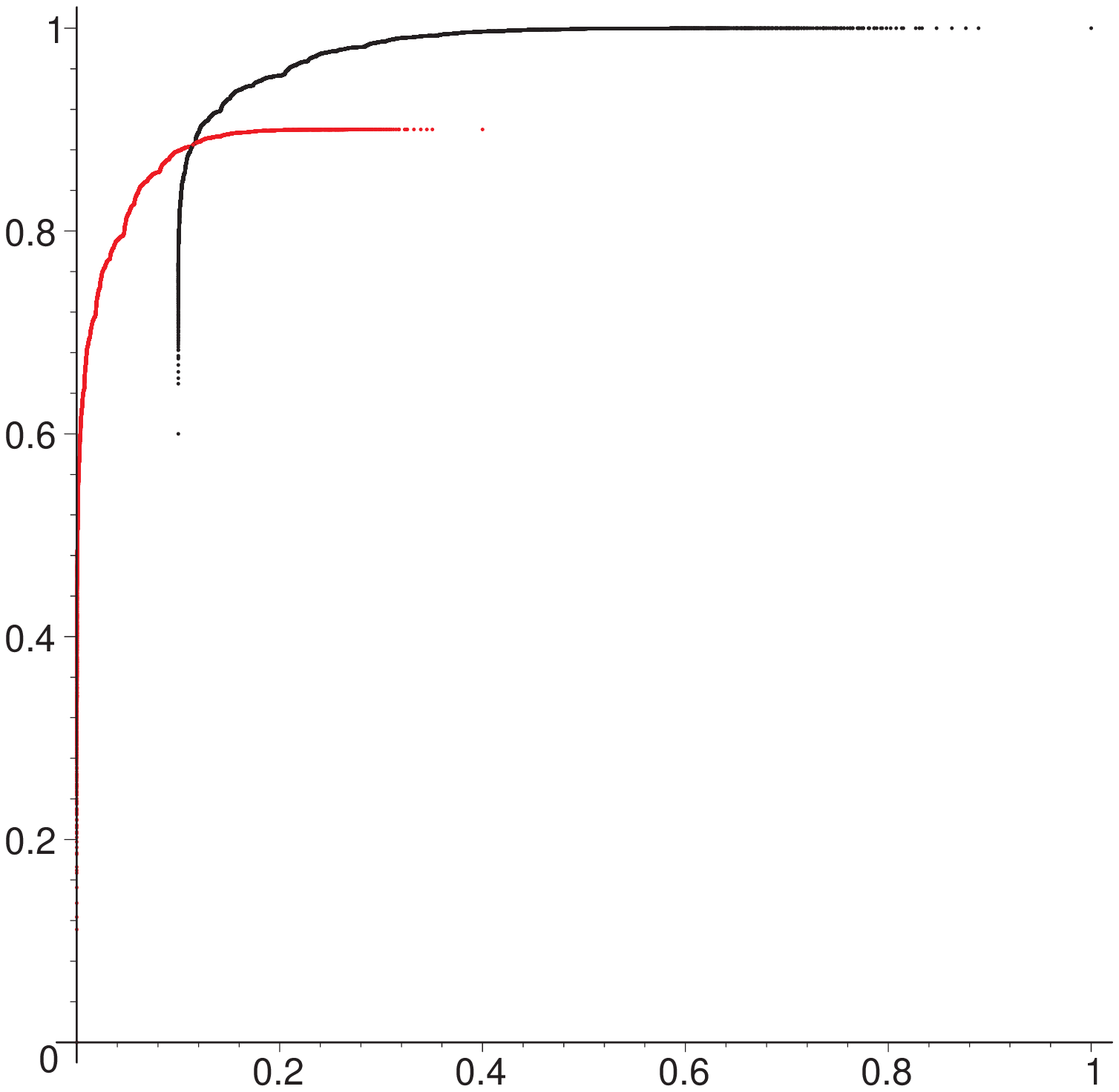}}
%\centering \scalebox{0.4} {\includegraphics{WMap_0409_zoom_2.eps}}
\centering \scalebox{0.4} {\includegraphics{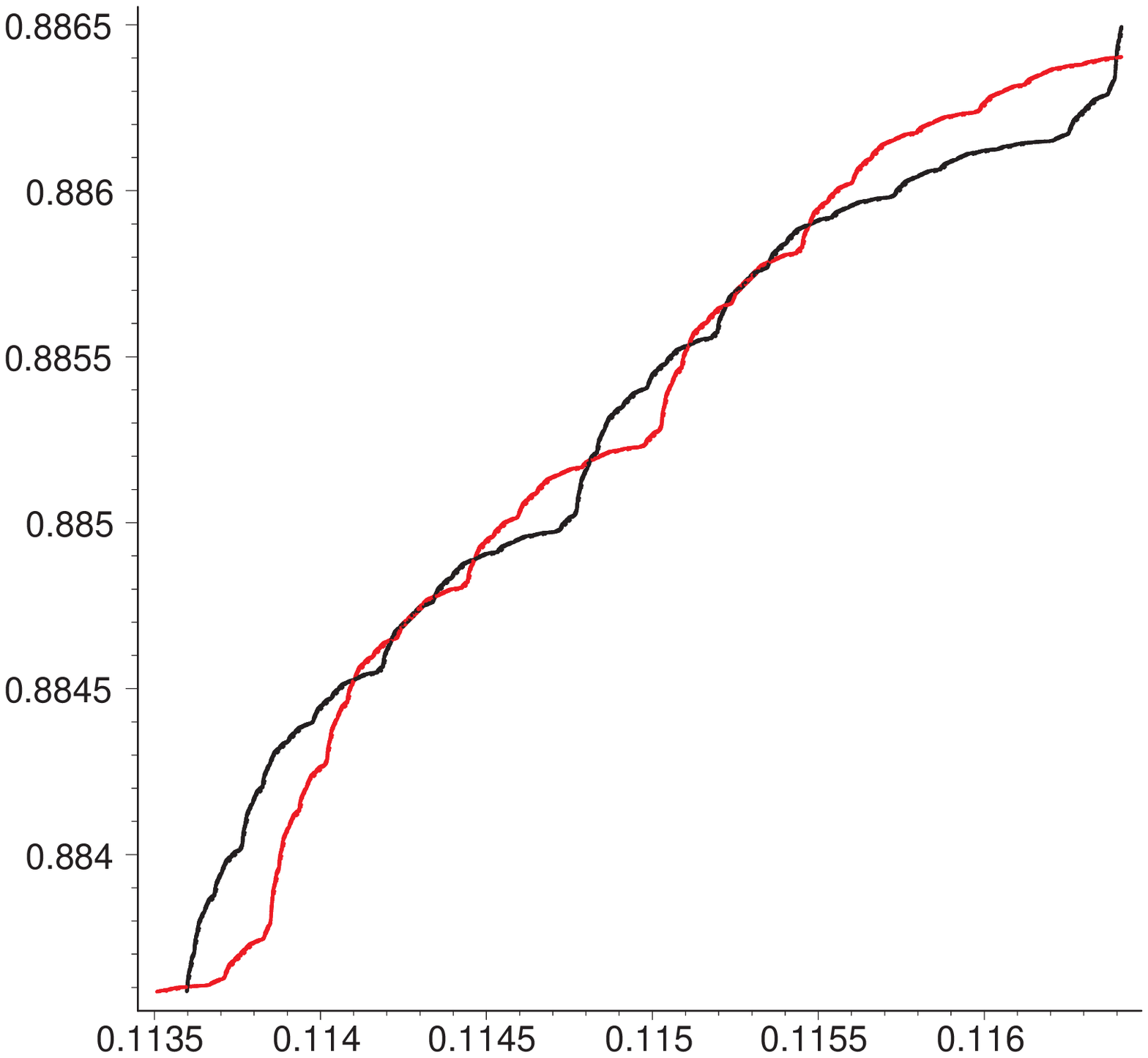}}
\caption{The set $T_0(B_{0.4,0.9})\cup T_1(B_{0.4,0.9})$, magnified in the neighbourhood of the intersection area.}\label{fig:neq}
\end{figure}

Theorem~\ref{thm1} is proved.

\section{Computational results on $G$}
\label{sec:comp}

We will first prove a special case, and then discuss how this can be extended.

Consider our example $\lambda = 0.4$ and $\mu = 0.9$ from before.
Consider an infinite word $a = (a_i)_{i=1}^\infty \in \{0, 1\}^\BbN$.
We define $pt_a$ as the limit
    $\lim_{n \to \infty} T_{a_1} \circ T_{a_2} \circ \dots \circ T_{a_n}$.
We note that the limit it independent of the point upon which we act.

Let $(x_1, y_1) = pt_{(01)^\infty}$ and $(x_2, y_y) = pt_{1 (0)^\infty}$.

We make two claims.
\begin{enumerate}
\item $(x_1, y_1) \in B_{0.4, 0.9}$.
\item $x_1 < x_2$ and $y_1 > y_2$.
\end{enumerate}

These two claims are sufficient to prove $(0.4, 0.9) \in G$.
To see this, we note that $B_{0.4, 0.9}$ is a continous increasing 
    function bounding $\partial_{top}(A_{0.4, 0.9})$ from below.

To see the first claim, we notice that
     \[ x_1 = \frac{\lambda(\mu-1)}{\lambda \mu -1} \text{ and }
        y_1 = \frac{\mu(\lambda-1)}{\lambda \mu -1} \]
We see that
    $x_1 + y_1 = \frac{2\lambda \mu - \lambda -\mu}{\lambda \mu -1}
    = 0.90625 < 1$.
We further see that
    \[ T_1(x_1, y_1) = \left(\frac{\mu-1}{\lambda\mu-1},
                             \frac{\lambda-1}{\lambda\mu-1}\right) \]
We further have that
    \[ \frac{\mu-1}{\lambda\mu-1} +  \frac{\lambda-1}{\lambda\mu-1}
       = \frac{\lambda+\mu-2}{\lambda \mu-1} = 1.09375 > 1 \]

We easily see that $(x_1, y_1) \in Y_0$, and by induction we have that
    $(x_1, y_1) \in Y_n$ for all $n$.
This proves that $(x_1, y_1) \in B_{0.4, 0.9} = \cap Y_n$.

The second claim follows as
\[
x_2  = 1-\mu = 0.1 >  x_1 = 0.0625 \] and
\[ y_2  = 1-\lambda = 0.4 <  x_1 = 0.84375. \]

We notice that the inequalities needed to ensure this result are true
    for more than this specific value of $\lambda$ and $\mu$.
In particular, so long as $pt_{(01)^\infty}$ is below the line $x+y=1$, and
    $T_1(pt_{(01)^\infty})$ is above the line $x+y=1$ we have that
    $pt_{(01)^\infty}$ is in $B_{\lambda, \mu}$.
Similar, the necessary inequality between $pt_{(01)^\infty}$ and
    $pt_{1(0)^\infty}$ can be easily checked for ranges of $\lambda$ and
    $\mu$.
For example, we can easily show a more general result that
    for all $(\lambda, \mu) \in [3/8, 7/16] \times [7/8, 15/16]$ that
    $pt_{(01)^\infty}$ is on $B_{\lambda, \mu}$ and that the necessarily
    in equality holds for $pt_{(01)^\infty}$ and
    $pt_{1 (0)^\infty}$.
That is, $[3/8, 7/16] \times [7/8, 15/16] \subset G$.

We computationally search for regions $R$ and eventually periodic
    $a \in \{0,1\}^{\BbN}$ such that
\begin{enumerate}
\item $pt_{a} \in B$ for $(\lambda, \mu) \in R$
\item $pt_{a}$ satisfies the desired inequality with one of
    $pt_{0(1)^\infty}$ or $pt_{1(0)^\infty}$.
\end{enumerate}

This data is collected on \cite{homepage}.

A graph of the proven regions is given in Figure \ref{fig:97}.
Each rectangle indicates a different region with a (potentially) different
    eventually periodic word $a$.
Some of these regions are very small, with a width of $1/2^{12}$.

\begin{figure}
\centering \scalebox{0.6} {\includegraphics{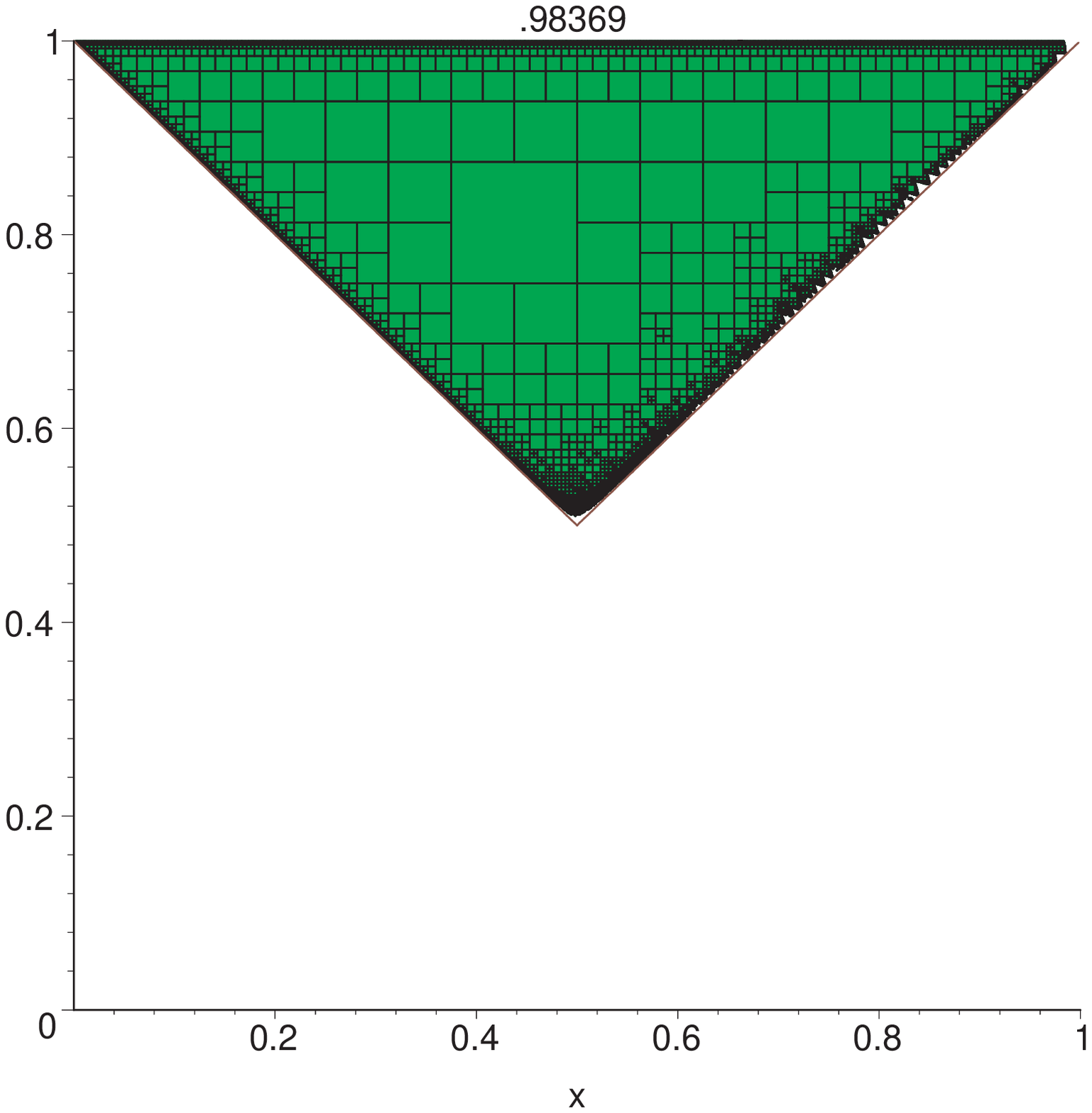}}
\caption{Region $G$}
\label{fig:97}
\end{figure}

\section{Proof of Theorem \ref{thm2}}
\label{sec:thm2}

Consider again our example with $\lambda = 0.4$ and $\mu =0.9$.
Let $(x_0,y_0) = \left(\frac{\lambda (\mu-1)}{\lambda \mu -1}, \frac{\mu(\lambda-1)}{\lambda \mu -1}\right)$,
  the solution to $T_0 T_1 (x_0, y_0) = (x_0, y_0)$.
Let $X = [x_0, 1] \times [y_0, 1]$.
Consider the sub-IFS generated by $\{T_0 T_1, T_1\}$.

It is easy to see that $T_0 T_1 (X) \subset X$,  $T_1 (X) \subset X$ and $T_0 T_1(X) \cap T_1(X) = \emptyset$.
Hence this sub-IFS satisfies the rectangular open set condition.
See Figure \ref{fig:ROS}.

\begin{figure}
\centering \scalebox{0.6} {\includegraphics{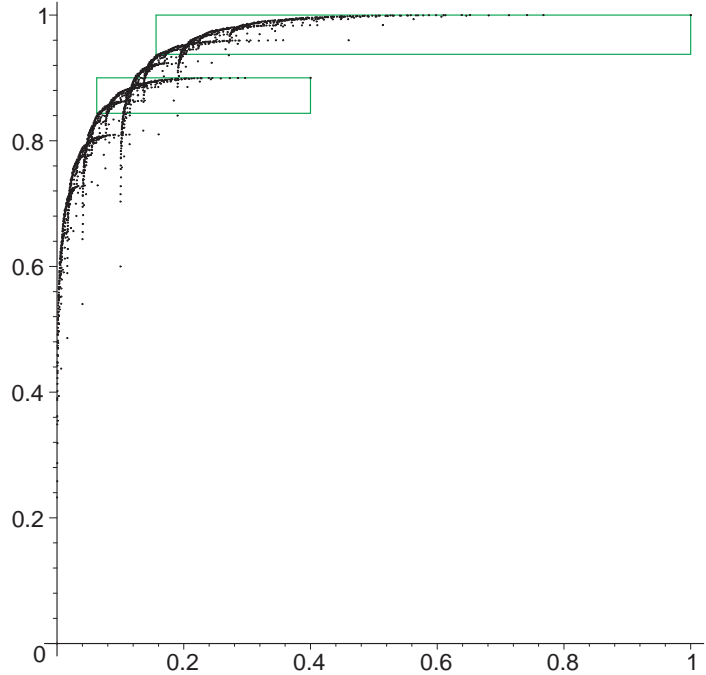}}
\caption{Rectangular Open Set Condition}
\label{fig:ROS}
\end{figure}

We further observe that the projection of this sub-IFS onto the $x$-axis is the interval $[x_0, 1]$, hence 
    dimension $1$.
Lastly, we see that the contractions are both of the form $(x, y) \to (a x + b, c x + d)$ where $a > c$.
Hence by Feng and Wang \cite{FW} we can compute the dimension for this sub-IFS.

In this case $s \approx 1.244273660$ which satisfies
    \[ (\lambda \mu)^s  + \mu^{s-1} \lambda = 1. \]

As the dimension of the full IFS is strictly less than 2, we see that it has no interior.  
Hence $\dim(K) = \dim(\partial(K)) \geq 1.244273660 > 1$.

More generally, let $w_1, w_2, \dots, w_n\in \{0,1\}^*$ such that 
    $|w_i|_1 \geq |w_i|_0$ for $i = 1, 2, \dots, n$, with at least
    one of the inequalities being strict.
Let $(x_i, y_i)$ be the fixed point of $T_{w_i}$ for $i = 1, \dots, n$.
Define $x_{\min} = \min (x_1, x_2, \dots, x_n)$, and similarly 
    $x_{\max}, y_{\min}$ and $y_{\max}$.
Define $X = [x_{\min}, x_{\max}] \times [y_{\min}, y_{\max}]$.
We see by construction that $T_{w_i}(X) \subset X$.
If we have that 
\begin{itemize}
\item $T_{w_i}(X) \cap T_{w_j}(X) = \emptyset$ for $i \neq j$,
\item The projection of the attractor of $\{T_{w_1}, T_{w_2}, \dots, T_{w_n}\}$
     onto the first coordinate is $[x_{\min}, x_{\max}]$,
\item The projections onto the first coordinate have non-trivial overlap for 
    some $T_{w_i}$ and $T_{w_j}$, $i \neq j$.
\end{itemize}
then the same argument will hold.
That is, by the rectangular open set condition $\{T_{w_1}, T_{w_2},  \dots,
    T_{w_n}\}$ has 
    dimension greater than $1$.
To see this we have from Feng and Wang the dimension satisfies
    \[ \sum b_i^{s-1} a_i = 1.\]
The left hand side is a decreasing function with respect to $s$ and evaluates
    to a value strictly greater than $1$ as $s = 1$, hence $s > 1$.

We use this argument with the sets
Extending this arguement to \[ \{T_0 T_1^m, T_1^n\}\hspace{0.5 in}\text{and}
    \hspace{0.5 in}\{T_0 T_1^m, T_1 T_0, T_1^2 T_0 \dots, T_1^n T_0\}. \] 
This covers greater than \Htrue\ of the parameter space.
See Figure \ref{fig:ROS2}.

\begin{figure}
\centering \scalebox{0.6} {\includegraphics{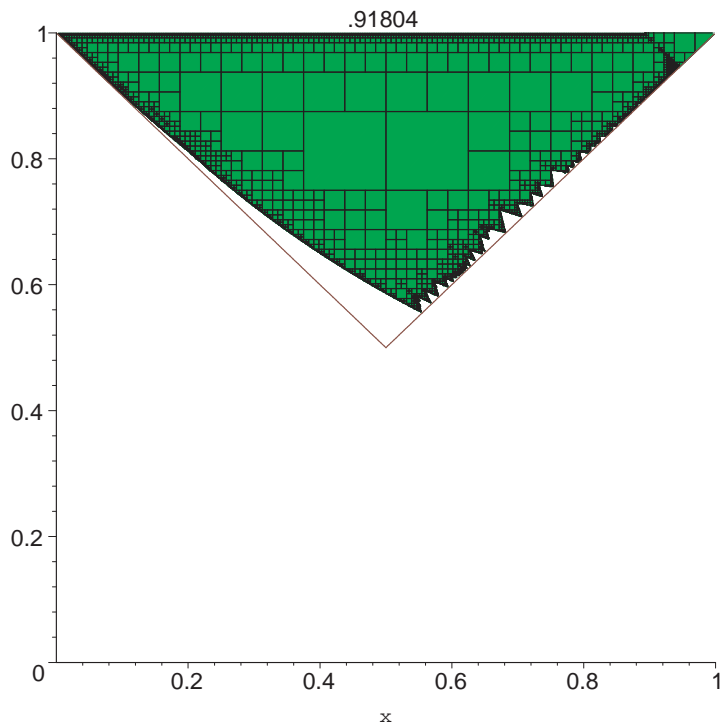}}
\caption{Rectangular Open Set Condition}
\label{fig:ROS2}
\end{figure}

One problem with this technique, is it doesn't seem to cover all cases.
In particular, for $\lambda = 0.45, \mu = 0.6$ we cannot find a combination
    of $m$ and $n$ such that $\{T_0 T_1^m, T_1^n\}$ or 
    $\{T_0 T_1^m, T_1 T_0, T_1^2 T_0 \dots, T_1^n T_0\}$ has the desired 
    properties.
We have also searched more generally for this particular case.
Letting $L$ being the set of all words $w$ of length up to $20$ where 
    $|w|_1 \geq |w|_0$,
    we have searched through all subsets of $L$ for possible proof using this 
    technique and found none.
Computationally, the dimension of $A_{0.45, 0.6}$ appears to be $1.08$.

Visually there seems to be a natural limit to these techniques, and a visible 
    gap between $\lambda+\mu=1$ and the cases that can be proved.

%\begin{figure}
%\centering \scalebox{0.6} {\includegraphics{PIC/ROS_0506.eps}}
%\caption{Rectangular Open Set Condition -- Almost solution}
%\label{fig:ROS2}
%\end{figure}

\section{Proof of Theorem \ref{thm3}}
\label{sec:thm3}

We will prove a more general result.

\begin{theorem}
  Let a two-dimensional IFS $\Phi$ be $\{T_0x=M_0x, T_1x=M_1 x+u\}$, where $M=M_0M_1=M_1M_0$. If $M$ is not scalar and $|\det M|\ge 1/\sqrt2$, then the attractor of $\Phi$ has a non-empty interior. If $M$ is scalar, the same result holds if $|\det M_0^2 M_1|\ge1/\sqrt2$.
\end{theorem}

\begin{proof}
  If $M$ is not scalar, then we consider the sub-IFS $\{T_0T_1,T_1T_0\}$. Both maps are given by the same matrix $M$, whence the first claim follows from the main result of our previous work \cite{HS-2D}.
  
 If $M$ is scalar, then we consider $\{T_0T_1T_0, T_0^2 T_1\}$, also with the same matrix and apply the same result.
\end{proof}

Return to Theorem~\ref{thm3}. Both matrices here are diagonal so they commute. Their product is scalar, so we apply the second case of the previous theorem. We thus get the condition $(\la\mu)^3\ge 1/\sqrt2$.

\end{document}